\documentclass[a4paper,11pt,draft]{amsart}
\usepackage{amssymb,amscd}
\usepackage[cp1251]{inputenc}
\usepackage[T2A]{fontenc}
\theoremstyle{plain}

\newtheorem{theo}{Theorem}[section]
\newtheorem{prop}[theo]{Proposition}
\newtheorem{lemm}[theo]{Lemma}
\newtheorem{coro}[theo]{Corollary}

\theoremstyle{definition}

\theoremstyle{remark}
\newtheorem*{rema}{Remark}

\numberwithin{equation}{section}

\newcommand{\mb}[1]{{\textbf {\textit#1}}}

\newcommand{\field}[1]{\mathbb{#1}}
\def\C{\field{C}}
\def\k{\mathbf{k}}
\newcommand{\R}{\field{R}}

\DeclareMathOperator{\Tor}{Tor}

\def\le{\leqslant}
\def\ge{\geqslant}

\def\As{\mbox{\it As}}

\newcommand{\zp}{\mathcal Z_P}

\DeclareMathOperator{\vc}{\mbox{\textit{vc}}}

\begin{document}

\title{Bigraded Betti numbers of some simple polytopes}

\author{Ivan Limonchenko}

\thanks{The work was supported by
grant MD-2253.2011.1 from the President of Russia.}

\address{Department of Geometry and Topology,
Faculty of Mathematics and Mechanics, Moscow State University,
Leninskiye Gory, Moscow 119992, Russia}
\email{iylim@mail.ru}


\begin{abstract}
The bigraded Betti numbers $\beta^{-i,2j}(P)$ of a simple polytope
$P$ are the dimensions of the bigraded components of the Tor
groups of the face ring $\k[P]$. The numbers $\beta^{-i,2j}(P)$
reflect the combinatorial structure of $P$ as well as the topology
of the corresponding moment-angle manifold~$\zp$, and therefore
they find numerous applications in combinatorial commutative
algebra and toric topology. Here we calculate some bigraded Betti
numbers of the type $\beta^{-i,2(i+1)}$ for associahedra, and
relate the calculation of the bigraded Betti numbers for
truncation polytopes to the topology of their moment-angle
manifolds. These two series of simple polytopes provide
conjectural extrema for the values of $\beta^{-i,2j}(P)$ among all
simple polytopes~$P$ with the fixed dimension and number of
facets.
\end{abstract}

\maketitle

\section{Introduction}

We consider \emph{simple convex $n$-dimensional polytopes} $P$ in
the Euclidean space $\R^n$ with scalar product
$\langle\;,\:\rangle$. Such a polytope $P$ can be defined as an
intersection of $m$ halfspaces:
\begin{equation}\label{ptope}
  P=\bigl\{\mb x\in\R^n\colon\langle\mb a_i,\mb
  x\rangle+b_i\ge0\quad\text{for }
  i=1,\ldots,m\bigr\},
\end{equation}
where $\mb a_i\in\R^n$, $b_i\in\R$. We assume that the hyperplanes
defined by the equations $\langle\mb a_i,\mb x\rangle+b_i=0$ are
in general position, that is, at most $n$ of them meet at a single
point. We also assume that there are no redundant inequalities
in~\eqref{ptope}, that is, no inequality can be removed
from~\eqref{ptope} without changing~$P$. Then $P$ has exactly $m$
\emph{facets} given by
$$
  F_i=\bigl\{\mb x\in P\colon\langle\mb a_i,\mb
  x\rangle+b_i=0\bigr\},\quad\text{for } i=1,\ldots,m.
$$

Let $A_P$ be the $m\times n$ matrix of row vectors $\mb a_i$, and
let $\mb b_P$ be the column vector of scalars $b_i\in\R$. Then we
can write~\eqref{ptope} as
\[
  P=\bigl\{\mb x\in\R^n\colon A_P\mb x+\mb b_P\ge\mathbf 0\},
\]
and consider the affine map
\[
  i_P\colon \R^n\to\R^m,\quad i_P(\mb x)=A_P\mb x+\mb b_P.
\]
It embeds $P$ into
\[
  \R^m_\ge=\{\mb y\in\R^m\colon y_i\ge0\quad\text{for }
  i=1,\ldots,m\}.
\]

Following \cite[Constr. 7.8]{B-P}, we define the space $\mathcal Z_P$
from the commutative diagram
\begin{equation}\label{cdiz}
\begin{CD}
  \mathcal Z_P @>i_Z>>\C^m\\
  @VVV\hspace{-0.2em} @VV\mu V @.\\
  P @>i_P>> \R^m_\ge
\end{CD}
\end{equation}
where $\mu(z_1,\ldots,z_m)=(|z_1|^2,\ldots,|z_m|^2)$. The latter
map may be thought of as the quotient map for the coordinatewise
action of the standard torus
\[
  \mathbb T^m=\{\mb z\in\C^m\colon|z_i|=1\quad\text{for }i=1,\ldots,m\}
\]
on~$\C^m$. Therefore, $\mathbb T^m$ acts on $\zp$ with quotient
$P$, and $i_Z$ is a $\mathbb T^m$-equivariant embedding.

By \cite[Lemma 7.2]{B-P}, $\mathcal Z_P$ is a smooth manifold of
dimension $m+n$, called the \emph{moment-angle manifold}
corresponding to~$P$.

Denote by $K_P$ the boundary $\partial P^*$ of the dual simplicial
polytope. It can be viewed as a simplicial complex on the set
$[m]=\{1,\ldots,m\}$, whose simplices are subsets
$\{i_1,\ldots,i_k\}$ such that $F_{i_1}\cap\ldots\cap
F_{i_k}\ne\varnothing$ in~$P$.

Let $\k$ be a field, let $\k[v_{1},\ldots,v_{m}]$ be the graded
polynomial algebra on $m$ variables, $\deg(v_{i})=2$, and let
$\Lambda[u_{1},\ldots,u_{m}]$ be the exterior algebra,
$\deg(u_{i})=1$. The \emph{face ring} (also known as the
\emph{Stanley--Reisner} \emph{ring}) of a simplicial complex $K$
on $[m]$ is the quotient ring
$$
   \k[K]=\k[v_{1},\ldots,v_{m}]/\mathcal I_K
$$
where $\mathcal I_K$ is the ideal generated by those square free
monomials $v_{i_{1}}\cdots{v_{i_{k}}}$ for which
$\{i_{1},\ldots,i_{k}\}$ is not a simplex in $K$. We refer to
$\mathcal I_{K}$ as the \emph{Stanley--Reisner ideal} of~$K$.

Note that $\k[K]$ is a module over $\k[v_{1},\ldots,{v_{m}}]$ via
the quotient projection. The dimensions of the bigraded components
of the $\Tor$-groups,
$$
  \beta^{-i,2j}(K):=\dim_{\k}\Tor^{-i,
  2j}_{\k[v_{1},\ldots,v_{m}]}\bigl(\k[K],\k\bigr),\quad
   0\le{i,j}\le{m},
$$
are known as the \emph{bigraded Betti numbers} of~$\k[K]$,
see~\cite{S} and \cite[\S3.3]{B-P}. They are important invariants
of the combinatorial structure of $K$. We denote
\[
  \beta^{-i,2j}(P):=\beta^{-i,2j}(K_{P}).
\]
The $\Tor$-groups and the bigraded Betti numbers acquire a
topological interpretation by means of the following result on the
cohomology of $\mathcal Z_P$:

\begin{theo}[{\cite[Theorem 8.6]{B-P} or \cite[Theorem 4.7]{P}}]\label{zkcoh}
The cohomology algebra of the moment-angle manifold $\zp$ is given
by the isomorphisms
\[
\begin{aligned}
  H^*(\zp;\k)&\cong\Tor_{\k[v_1,\ldots,v_m]}(\k[K_ P],\k)\\
  &\cong H\bigl[\Lambda[u_1,\ldots,u_m]\otimes \k[K_ P],d\bigr],
\end{aligned}
\]
where the latter algebra is the cohomology of the differential
bigraded algebra whose bigrading and differential are defined by
\[
  \mathop{\mathrm{bideg}} u_i=(-1,2),\;\mathop{\mathrm{bideg}} v_i=(0,2);\quad
  du_i=v_i,\;dv_i=0.
\]
\end{theo}
Therefore, cohomology of $\zp$ acquires a bigrading and the
topological Betti numbers $b^{q}(\zp)=\dim_{k}H^{q}(\zp;\k)$
satisfy
\begin{equation}\label{bp}
  b^{q}(\zp)=\sum\limits_{-i+2j=q}\beta^{-i,2j}(P).
\end{equation}

Poincar\'e duality in cohomology of $\zp$ respects the bigrading:
\begin{theo}[{\cite[Theorem 8.18]{B-P}}]\label{Poinc}
The following formula holds:
$$
   \beta^{-i,2j}(P)=\beta^{-(m-n)+i,2(m-j)}(P).
$$
\end{theo}

From now on we shall drop the coefficient field $\k$ from the
notation of (co)homology groups. Given a subset $I\subset{[m]}$,
we denote by $K_{I}$ the corresponding \emph{full subcomplex} of
$K$ (the restriction of $K$ to $I$). The following classical
result can be also obtained as a corollary of Theorem~\ref{zkcoh}:
\begin{theo}[Hochster, see~{\cite[Cor. 8.8]{B-P}}]\label{hoh}
Let $K=K_P$. We have:
$$
   \beta^{-i,2j}(P)=\sum\limits_{J\subset{[m]},|J|=j}\dim
   \widetilde{H}^{j-i-1}(K_{J}).
$$
\end{theo}

We also introduce the following subset in the boundary of~$P$:
\begin{equation}\label{PI}
  P_I=\bigcup_{i\in I}F_i\subset P.
\end{equation}
Note that if $K=K_P$ then $K_I$ is a deformation retract of~$P_I$
for any~$I$. The following is a direct corollary of
Theorem~\ref{hoh}.

\begin{coro}\label{fact1} We have
\[
  \beta^{-i,2(i+1)}(P)=\sum_{I\subset[m],|I|=i+1}\bigl(cc(P_{I})-1\bigr),
\]
where $cc(P_I)$ is the number of connected components of the
space~$P_I$.
\end{coro}

The structure of this paper is as follows. Calculations for
Stasheff polytopes (also known as associahedra) are given in
Section~2. In Section~3 we calculate the
bigraded Betti numbers of truncation polytopes (iterated vertex
cuts of simplices) completely. These calculations were first made
in~\cite{T-H} using a similar but slightly different method; an
alternative combinatorial argument was given in~\cite{Ch-K}. We
also compare the calculations of the Betti numbers with the known
description of the diffeomorphism type of $\zp$ for truncation
polytopes~\cite{B-M}.

The author is grateful to his scientific adviser Taras Panov for
fruitful discussions and advice which was always so kindly
proposed during this work.

\section{Stasheff polytopes}
\emph{Stasheff polytopes}, also known as \emph{associahedra}, were
introduced as combinatorial objects in the work of Stasheff on
higher associativity~\cite{stas63}. Explicit convex realizations
of Stasheff polytopes were found later by Milnor and others,
see~\cite{B} for details.

We denote the $n$-dimensional
Stasheff polytope by~$\As^n$. The $i$-dimensional faces of $\As^n$
($0\le{i}\le{n-1}$) bijectively correspond to the sets of $n-i$
pairwise nonintersecting diagonals in an~$(n+3)$-gon~$G_{n+3}$.
(We assume that diagonals having a common vertex are
nonintersecting.) A face $H$ belongs to a face $H'$ if and only if
the set of diagonals corresponding to $H$ contains the set of
diagonals corresponding to~$H'$.

In particular, vertices of $As^{n}$ correspond to complete
triangulations of $G_{n+3}$ by its diagonals, and facets of
$As^{n}$ correspond to diagonals of~$G_{n+3}$. We therefore
identify the set of diagonals in $G_{n+3}$ with the set of facets
$\{F_1,\ldots,F_m\}$ of~$\As^n$, and identify both sets with $[m]$
when it is convenient. Note that $m=\frac{n(n+3)}2$.

We shall need a convex realization of $\As^n$
from~\cite[Lecture~II,~Th.~5.1]{B}:

\begin{theo}
$\As^n$ can be identified with the intersection of the
parallelepiped
\[
  \bigl\{\mb y\in\R^n\colon 0\le y_j\le j(n+1-j)\quad
  \text{for }\;1\le j\le n\bigr\}
\]
with the halfspaces
\[
  \bigl\{\mb y\in\R^n\colon y_j-y_k+(j-k)k\ge0 \bigr\}
\]
for $1\le k<j\le n$.
\end{theo}


\begin{prop}
We have:
$$
  b^{3}(\mathcal Z_{As^{n}})=\beta^{-1,4}(\As^n)=\binom{n+3}{4}.
$$
\end{prop}
\begin{proof}
The number $\beta^{-1,4}(P)$ is equal to the number of monomials
$v_{i}v_{j}$ in the Stanley--Reisner ideal of
$P$~\cite[\S3.3]{B-P}, or to the number of pairs of disjoint
facets of~$P$. In the case $P=\As^n$ the latter number is equal to
the number of pairs of intersecting diagonals in the
$(n+3)$-gon~$G_{n+3}$, see~\cite[Lecture II, Cor 6.2]{B}. It
remains to note that, for any 4-element subset of vertices of
$G_{n+3}$ there is a unique pair of intersecting diagonals whose
endpoints are these 4 vertices.
\end{proof}

\begin{rema}The above calculation can be also made using the general formula
$\beta^{-1,4}(P)=\binom{f_{0}}{2}-f_{1}$,
see~\cite[Lemma~8.13]{B-P}, where $f_{i}$ is the number of
$(n-i-1)$-faces of~$P$. The numbers $f_{i}$ for $\As^n$ are
well-known, see~\cite[Lecture II]{B}.
\end{rema}

In what follows, we assume that there are no multiple intersection points of the
diagonals of $G_{n+3}$, which can be achieved by a small
perturbation of the vertices. We choose a cyclic order of vertices
of $G_{n+3}$, so that 2 consequent vertices are joined by an edge.
We refer to the diagonals of $G_{n+3}$ joining the $i$th and the
$(i+2)$th vertices (modulo $n+3$), for $i=1,\ldots,n+3$ as
\emph{short}; other diagonals are \emph{long}.

We refer to intersection points of diagonals inside $G_{n+3}$ as
\emph{distinguished points}. A diagonal segment joining two
distinguished points is called a \emph{distinguished segment}.
Finally, a \emph{distinguished triangle} is a triangle whose
vertices are distinguished points and whose edges are
distinguished segments.

\begin{theo}\label{TH1}
We have:
$$
   b^{4}(\mathcal Z_{As^{n}})=\beta^{-2,6}(\As^{n})=5\binom{n+4}{6}
$$
\end{theo}
\begin{proof}

We need to calculate the number of generators in the 4th
cohomology group of
$H[\Lambda[u_{1},\ldots,u_{m}]\times{k[\As^{n}]},d]$, see
Theorem~\ref{zkcoh} (note that here $m=\frac{(n+3)n}{2}$ is the
number of diagonals in $G_{n+3}$). This group is generated by the
cohomology classes of cocycles of the type $u_{i}u_{j}v_{k}$,
where $i\ne{j}$ and $u_{i}v_{k}$, $u_{j}v_{k}$ are 3-cocycles.
These 3-cocycles correspond to the pairs $\{i,k\}$ and $\{j,k\}$
of intersecting diagonals in $G_{n+3}$, or to a pair of
distinguished points on the $k$th diagonal. It follows that every
cocycle $u_{i}u_{j}v_{k}$ is represented by a distinguished
segment. The identity
$$
  d(u_{i}u_{j}u_{k})=u_{i}u_{j}v_{k}-u_{i}v_{j}u_{k}+v_{i}u_{j}u_{k}
$$
implies that the cohomology classes represented by the cocycles in
the right hand side are linearly dependent. Every such identity
corresponds to a distinguished triangle.

We therefore obtain that $\beta^{-2,6}(\As^{n})=S_{n+3}-T_{n+3}$
where $S_{n+3}$ is the number of distinguished segments and
$T_{n+3}$ is the number of distinguished triangles inside $G_{n+3}$.
These numbers are calculated in the next three lemmas.

\begin{lemm}
The number of distinguished triangles in $G_{n+3}$ is given by
$$
   T_{n+3}=\binom{n+3}{6}
$$
\end{lemm}
\begin{proof}
We note that there is only one distinguished triangle in a hexagon
(see Fig.~\ref{hexfig}); and therefore every 6 vertices of
$G_{n+3}$ contribute one distinguished triangle.
\begin{figure}[h]
\begin{center}
\begin{picture}(120,40)

\put(45,0){\line(1,0){30}}
\put(75,0){\line(2,1){20}}
\put(95,10){\line(-2,3){20}}
\put(75,40){\line(-1,0){30}}
\put(45,40){\line(-2,-3){20}}
\put(25,10){\line(2,-1){20}}

\put(45,0){\line(0,1){40}}
\put(45,0){\line(3,4){30}}
\put(45,0){\line(5,1){50}}

\put(75,0){\line(-5,1){50}}
\put(75,0){\line(-3,4){30}}
\put(75,0){\line(0,1){40}}

\put(95,10){\line(-1,0){70}}
\put(95,10){\line(-5,3){50}}

\put(75,40){\line(-5,-3){50}}

\multiput(67.2,10)(0.1,0){6}{\line(-3,4){7.5}}

\multiput(52.5,10)(0.1,0){6}{\line(3,4){7.5}}

\linethickness{0.6mm} \put(67.5,10){\line(-1,0){15}}

%

\end{picture}
\caption{} \label{hexfig}
\end{center}
\end{figure}
\end{proof}

Given a diagonal $d$ of $G_{n+3}$, denote by $p(d)$ the number of
distinguished points on~$d$. We define the \emph{length} of $d$ as
the smallest of the numbers of vertices of $G_{n+3}$ in the open
halfplanes defined by~$d$. Therefore, short diagonals have length
1 and all diagonals have length $\le{\frac{n+1}{2}}$. We refer to
diagonals of maximal length simply as \emph{maximal}. Obviously
$p(d)$ depends only on the length of~$d$, and we denote by $p(j)$
the number of distinguished points on a diagonal of length~$j$.

\begin{lemm}\label{l3}
If $n=2k-1$ is odd, then
$$
   S_{n+3}=\frac{n+3}{2}\sum\limits_{l=1}^{k-1}
   \Bigl(4l^{2}k^{2}-2k(2l^{3}+l)\Bigr)+\frac{n+3}{4}k^{2}(k^{2}-1).
$$
If $n=2k-2$ is even, then
$$
  S_{n+3}=\frac{n+3}{2}\sum\limits_{l=1}^{k-1}
  \Bigl(4l^{2}k^{2}-2k(2l^{3}+2l^{2}+l)+(l^{4}+2l^{3}+2l^{2}+l)\Bigr).
$$
\end{lemm}
\begin{proof}
First assume that $n=2k-1$. Then
\begin{align*}
  S_{n+3}&=\sum\limits_{d}\frac{p(d)(p(d)-1)}{2}=\\
  &=(n+3)\biggl(\sum\limits_{j=1}^{\frac{n+1}{2}}\frac{p(j)(p(j)-1)}{2}\biggr)-
  \biggl(\frac{n+3}{2}\biggr)\frac{p(\frac{n+1}{2})(p(\frac{n+1}{2})-1)}{2},
\end{align*}
since the number of distinguished segments on the maximal
diagonals is counted in the sum twice.

We denote by $v$ the $(n+3)$th vertex of $G_{n+3}$ and numerate
the diagonals coming from $v$ by their lengthes. We denote by
$c(i,j)$ the number of intersection points of the $j$th diagonal
coming from $v$ with the diagonals from the $i$th vertex, for
$1\le{i}\le{j}\le{\frac{n+1}{2}}$, and set $c(i,j)=0$ for $i>j$.
Then we have
\begin{equation}\label{pj}
   p(j)=\sum\limits_{i=1}^{\frac{n+1}{2}}c(i,j),
\end{equation}
To compute $c(i,j)$ we note that
\begin{align*}
&c(1,1)=n;\\
&c(i,j-1)=c(i,j)+1\quad\text{for }1\le{i}<j\le{\frac{n+1}{2}};\\
&c(i+1,j+1)=c(i,j)-1\quad\text{for }1\le{i}\le
j\le{\frac{n-1}{2}}.
\end{align*}

It follows that
\begin{equation}\label{cij}
  c(i,j)=c(1,j-i+1)-(i-1)=c(1,1)-(j-i)-(i-1)=n-j+1,
\end{equation}
for $i\le{j}$. Note that $c(i,j)$ does not depend on $i$. Substituting this in~\eqref{pj} and then
substituting the resulting expression for $p(j)$ in the sum for $S_{n+3}$ above we
obtain the required formula.

The case $n=2k-2$ is similar. The only difference is that there
are two maximal diagonals coming from every vertex of $G_{n+3}$,
so that no subtraction is needed in the sum for $S_{n+3}$.
\end{proof}

\begin{lemm}\label{l4}
The number of distinguished segments is given by
$$
   S_{n+3}=(n+3)\binom{n+3}{5}.
$$
\end{lemm}
\begin{proof}
This follows from Lemma~\ref{l3} by summation using the following
formulae for the sums $\Sigma_{n}$ of the $n$th powers of the
first $(k-1)$ natural numbers:
\begin{align*}
\Sigma_{1}&=\frac{k(k-1)}{2},
&\Sigma_{2}&=\frac{k(k-1)(2k-1)}{6},\\
\Sigma_{3}&=\frac{k^{2}(k-1)^{2}}{4},
&\Sigma_{4}&=\frac{k(k-1)(2k-1)(3k^{2}-3k-1)}{30}.
\end{align*}
\end{proof}

Now Theorem~\ref{TH1} follows from Lemma~\ref{l3} and Lemma~\ref{l4}.
\end{proof}

The following fact follows from the description of the
combinatorial structure of~$\As^n$ (see also~\cite[Lecture II, Cor.
6.2]{B}):
\begin{prop}\label{fact2}
Two facets $F_ 1$ and $F_ 2$ of the polytope $\As^n$ do not
intersect if and only if the corresponding diagonals $d_1$ and
$d_2$ of the polygon $G_{n+3}$ intersect (in a distinguished point).
\end{prop}

\begin{lemm}\label{le0}
The number of distinguished points on a maximal diagonal of
$G_{n+3}$ is given by
\begin{align*}
&q=q(n)=\begin{cases}
  \frac{n(n+2)}{4},&\text{if $n$ is even;}\\
  \frac{(n+1)^2}{4},&\text{if $n$ is odd.}
  \end{cases}
\end{align*}
\end{lemm}
\begin{proof}
The case $n=2$ is obvious. If $n$ is odd, then setting
$j=\frac{n+1}{2}$ in~\eqref{pj} and using~\eqref{cij} we calculate
$$
  p(j)=\sum\limits_{i=1}^{\frac{n+1}{2}}c\Bigl(i,\frac{n+1}{2}\Bigr)=
  \frac{(n+1)^{2}}{4}.
$$
If $n$ is even, then the maximal diagonal has length $j=\frac{n}{2}$. It is easy to see
that we have $p(j)=\sum_{i=1}^{n/2}c(i,j)$ instead of~\eqref{pj},
and~\eqref{cij} still holds. Therefore,
$$
   p(j)=\sum\limits_{i=1}^{\frac{n}{2}}c\Bigr(i,\frac{n}{2}\Bigr)=
   \frac{n(n+2)}{4}.
$$
\end{proof}

\begin{theo}\label{mainass}
Let $P=\As^n$ be an $n$-dimensional associahedron, $n\ge{3}$. The
bigraded Betti numbers of $P$ satisfy
\begin{align*}
  &\beta^{-q,2(q+1)}(P)=\begin{cases}
  n+3,&\text{if $n$ is even;}\\
  \frac{n+3}{2},&\text{if $n$ is odd;}
  \end{cases}\\
  &\beta^{-i,2(i+1)}(P)=0\quad\text{for }i\ge{q+1};
\end{align*}
where $q=q(n)$ is given in Lemma~\ref{le0}.
\end{theo}
\begin{proof}
We prove the theorem by induction on $n$. The base case $n=3$ can
be seen from the tables of bigraded Betti numbers below. By
Corollary~\ref{fact1}, in order to calculate
$\beta^{-i,2(i+1)}(P)$, we need to find all $I\subset[m]$, \
$|I|=i+1$, whose corresponding $P_I$ has more than one connected
component. In the case $i=q$ we shall prove that $cc(P_I)\le2$ for
$|I|=q+1$, and describe explicitly those $I$ for which
$cc(P_I)=2$. In the case $i>q$ we shall prove that $cc(P_I)=1$ for
$|I|=i+1$. These statements will be proven as separate lemmas; the
step of induction will follow at the end.

We numerate the vertices of $G_{n+3}$ by the integers from 1 to
$n+3$. Then every diagonal $d$ corresponds to an ordered pair
$(i,j)$ of integers such that $i<j-1$. It is convenient to view
the diagonal corresponding to $(i,j)$ as the segment $[i,j]$
inside the segment $[1,n+3]$ on the real line. Then
Proposition~\ref{fact2} may be reformulated as follows:

\begin{prop}\label{inter}
The  facets $F_{1}$ and $F_{2}$ of $P=\As^n$ do not intersect if
and only if the corresponding segments $[i_1,j_1]$ and $[i_2,j_2]$
overlap, that is,
$$
  F_{1}\cap{F_{2}}=\varnothing\quad\Longleftrightarrow\quad
  i_{1} <  i_{2} <  j_{1}
 <  j_{2}\quad \text{or}\quad i_{2} <  i_{1} <  j_{2} <  j_{1}.
$$
\end{prop}

Let $I$ be a set of diagonals of $G_{n+3}$ (or integer segments
in~$[1,n+3]$), and $P_I$ the corresponding set~\eqref{PI}. We
write $I=I_1\sqcup I_2$ whenever $P_I$ has exactly two connected
components corresponding to $I_1$ and $I_2$. We also denote by
$e(I)$ the set of endpoints of segments from~$I$; its a subset of
integers between 1 and $n+3$.

\begin{prop}\label{ei1i2}
If $I=I_1\sqcup I_2$ then the subsets $e(I_1)$ and $e(I_2)$ are
disjoint.
\end{prop}
\begin{proof}
Follows directly from Proposition~\ref{inter}.
\end{proof}

Given an integer $m\in[1,n+3]$ and a set of segments~$I$, we
denote by $c_I(m)$ the number of segments in $I$ that have $m$ as
one of their endpoints (equivalently, the number of diagonals in
$I$ with endpoint~$m$). Then $0\le{c_I(m)}\le{n}$.

\begin{prop}\label{cm}
If $I=I_{1}\sqcup{I_{2}}$ then there exists $m$ such that
$c_I(m)\le{\frac{n+1}{2}}$.
\end{prop}
\begin{proof}
Assume the opposite is true. Choose integers $m_1\in e(I_1)$ and
$m_2\in e(I_2)$. Since $c_I(m_1)>\frac{n+1}{2}$, \
$c_I(m_2)>\frac{n+1}{2}$ and $e(I_1)$, $e(I_2)$ are disjoint by
the previous proposition, we obtain that the total number of
elements in $e(I)$ is more than
$2+\frac{n+1}{2}+\frac{n+1}{2}=n+3$. A contradiction.
\end{proof}

\begin{lemm}\label{le1} We have that $cc(P_I)\le{2}$
for $|I|>l(n)=\frac{n(n+2)}{4}$.
\end{lemm}
\begin{proof}
We prove this lemma by induction on $n$.

First let $n=3$, and assume that the statement of the lemma fails,
i.e. there is a set
$I=I_{1}\sqcup{I_{2}}\sqcup{I_{3}}\sqcup{\ldots}$ of diagonals of
$G_{6}$, \ $|I|\ge{4}$, such that $cc(P_ I)\ge{3}$. As there are
only 3 long diagonals in $G_{6}$, there exists a short diagonal
$d\in{I}$; assume $d\in{I_{1}}$. Since $cc(P_ I)\ge{3}$, every
$e\in{I_{2}}$ and $f\in{I_{3}}$ intersect~$d$. Hence, $e$ and $f$
meet at a vertex $A$ of~$G_{6}$.
This contradicts the fact that $e(I_{2})$ and $e(I_{3})$ are
disjoint (see~Proposition~\ref{ei1i2}).

Now let $n>3$ and assume that there is a set
$I=I_{1}\sqcup{I_{2}}\sqcup{I_{3}}\sqcup{\ldots}$ of diagonals of
$G_{n+3}$, $|I|>\frac{n(n+2)}{4}$, with $cc(P_{I})\ge{3}$.

If there exists $m\in[1,n+3]$ with $c_I(m)=0$, then we may assume
that $m$ is the first vertex, and view $I$ as a set of diagonals
of $G_{n+2}$ (the segment $[2,n+3]$ cannot belong to $I$, since
otherwise $cc(P_I)=1$). As $l(n)>l(n-1)$, the induction assumption
finishes the proof of the lemma.

Now $c_I(m)\ge{1}$ for every $m\in{[1,n+3]}$. Then by the argument
similar to that of Proposition~\ref{cm}, there exists $m$ with
$c_I(m)\le{\frac{n}{3}}$. Consider 2 cases:

\noindent{\bf1.} There exists $m_{0}\in{e(I_{k})}$ for some
$1\le{k}\le{cc(P_ I)}$ with the smallest value of
$c_I(m)\le{\frac{n}{3}}$, such that $|I_{k}|>c_I(m_{0})$.

We may assume that one of these $m_{0}$ is the first vertex.
Removing from $I$ all segments with endpoint~1, we obtain a new
set $\tilde{I}$ of segments inside $[2,n+3]$ (the segment
$[2,n+3]$ cannot belong to $I$, as otherwise $cc(P_ I)\le{2}$). We
have:
\[
  |\tilde{I}|=|I|-c_I(1)>\frac{n(n+2)}{4}-\frac{n}{3}>
  \frac{(n-1)(n+1)}{4}=l(n-1).
\]
By the induction assumption, $2\ge cc(P_{\tilde{I}})\ge cc(P_
I)\ge 3$. A contradiction.

\noindent{\bf2.} For every vertex $m_{0}$ with the smallest value
of $c_I(m)\ge{1}$ we have $|I_{k}|=c_I(m_{0})$, where
$m_{0}\in{e(I_{k})}$.

Again, we may assume that one of these $m_{0}$ is the first vertex
$1\in I_k$. We have $c_I(1)=1$, as otherwise there are $\ge{2}$
integer points $m$ inside $[2,n+3]$ which belong to $e(I_{k})$ and
have $c_I(m)=1$ (remember that $|I_{k}|=c_I(m_{0})$).

Without loss of generality we may assume that $k=1$. Then
$$
  |I|=1+|I_{2}|+|I_{3}|+\ldots \le
  1+(1+q(n-1))\le{2+\frac{n^{2}}{4}}\le{\frac{n(n+2)}{4}}.
$$
The first inequality above holds since
$\tilde{I}={I_{2}}\sqcup{I_{3}}\sqcup{\ldots}$ is a set of
diagonals of $G_{n+2}$ (the segment $[2,n+3]$ cannot belong to
$I$, because $cc(P_ I)\ge{3}$), and we can apply to $\tilde{I}$
the induction assumption in the proof of the main
Theorem~\ref{mainass}, which gives us $|\tilde I|\le1+q(n-1)$. We
get a contradiction with the assumption $|I|>\frac{n(n+2)}{4}$.
\end{proof}

\begin{lemm}\label{i1i2}
Assume that $I=I_1\sqcup I_2$, \ $|I|\ge q+1$, \ $|I_1|\ge2$ and
$|I_2|\ge2$. Then there exists another $I'$ such that
$I'=I'_{1}\sqcup{I'_{2}}$, \ $|I'_{1}|=1$ and $|I'|>|I|$.
\end{lemm}
\begin{proof}
The proof is by induction on $n$. The cases $n=3, 4, 5$ are checked
by a direct computation (see also the tables at the end of this
section).

Changing the numeration of vertices of $G_{n+3}$ if necessary, we
may assume that the first vertex has the smallest value
of~$c_I(m)$. Then $c_I(1)\le{\frac{n+1}{2}}$ by
Proposition~\ref{cm}.  Without loss of generality we may assume
that $1\notin e(I_1)$.

We claim that the segment $[2,n+3]$ does not belong to $I$.
Indeed, in the opposite case $c_I(1)>0$ (otherwise $cc(P_I)=1$),
$1\in{e(I_{2})}$, $[2,n+3]\in{I_{1}}$. If  $c_I(1)\ge{2}$, then
there is an integer point $m\in e(I_{2})$ inside $[2,n+3]$ with
$c_I(m)=1<c_I(1)$, which contradicts the choice of the first
vertex. Then $c_I(1)=1$ and $[2,n+3]\in{I_{1}}$ imply that
$|I_{2}|=c_I(1)=1$ which contradicts the assumption
$|I_{2}|\ge{2}$ in the lemma.

Removing from $I$ all segments with endpoint~1, we obtain a new
set $\tilde{I}$ of integer segments inside $[2,n+3]$. Note that
\begin{equation}\label{tildei}
  |\tilde{I}|=|I|-c_I(1)\ge|I|-{\textstyle\bigl[\frac{n+1}2\bigr]}.
\end{equation}

We want to apply the induction assumption to the set $\tilde{I}$
of integer segments inside $[2,n+3]$, viewed as diagonals in an
$(n+2)$-gon~$G_{n+2}$. To do this, we need to check the
assumptions of the lemma for $\tilde I$.

First, we claim that $\tilde{I}=\tilde{I_{1}}\sqcup\tilde{I_{2}}$,
i.e. $P_{\tilde I}$ has exactly two connected components. Indeed,
it obviously has at least two components, and the number of
components cannot be more than two by Lemma~\ref{le1}, since
\[
  |\tilde{I}|\ge|I|-\frac{n+1}2\ge q+1-\frac{n+1}2>
  \frac{(n+1)^{2}}{4}-\frac{n+1}{2}=l(n-1).
\]
Second, $|\tilde{I_{1}}|=|I_{1}|\ge{2}$ and
$|I_{2}|\ge{|\tilde{I_{2}}|}\ge{1}$. If $|\tilde{I_{2}}|=1$ then
we have either $c_I(1)=1$ or $c_I(1)=2$. (Indeed, if $c_I(1)=0$
then $|I_2|=|\tilde{I_{2}}|=1$, which contradicts the assumption,
and $c_I(1)$ cannot be more than 2 as otherwise $c_I(1)$ is not
the smallest one.) Therefore, $|I_{2}|\le3$. We also have
$|I_{1}|=|\tilde{I_{1}}|\le{p(d)}$,where
$d\in\tilde{I_{2}}=\{d\}$, because $d$ intersects every diagonal
from $I_{1}$. Due to Lemma~\ref{le0},
$p(d)\le{q(n-1)}\le{\frac{n^{2}}{4}}$. Hence,
\[
  |I|=|I_{1}|+|I_{2}|\le{p(d)+3}\le\frac{n^{2}}{4}+3\le{\frac{(n+1)^{2}}{4}}<q(n)+1\le{|I|}
\]
for $n\ge{6}$. A contradiction. Thus, $|\widetilde{I_2}|\ge2$.

It remains to check that $|\tilde{I}|\ge{q(n-1)+1}$. If $n$ is
odd, then
\[
  |\tilde{I}|\ge{|I|-\frac{n+1}{2}}\ge
  {\frac{(n+1)^{2}}{4}+1-\frac{n+1}{2}}=
  \frac{(n-1)(n+1)}{4}+1=q(n-1)+1.
\]
If $n$ is even, then
\[
  |\tilde{I}|\ge{|I|-\frac{n}{2}}\ge
  {\frac{n(n+2)}{4}+1-\frac{n}{2}}=\frac{n^{2}}{4}+1=q(n-1)+1.
\]

Now, applying the induction assumption to $\tilde I$, we find a
new set of integer segments $\tilde{J}$ inside $[2,n+3]$ with
$|\tilde{J}|>|\tilde{I}|$ and $|\tilde{J_{1}}|=1$. Then
$\tilde{J_{1}}=\{d\}$, where $d$ is a diagonal of $G_{n+2}$.
Hence, $|\tilde{J}|=|\tilde{J_{1}}|+|\tilde{J_{2}}|\le{1+p(d)}$.
We have $p(d)\le{q(n-1)}$, and the equality holds if and only if
$d=d_{max}$ is a maximal diagonal in~$G_{n+2}$. Therefore, we can
replace $\tilde J$ by $J'=J'_{1}\sqcup{J'_{2}}$, where
$J'_{1}=\{d_{max}\}$ and $J'_{2}$ is the set of diagonals in
$G_{n+2}$ which intersect $d_{max}$ at its distinguished points.
Indeed, we have
\begin{equation}\label{ji}
|J'|=1+q(n-1)\ge{1+p(d)}\ge{|\tilde{J}|}>|\tilde{I}|.
\end{equation}

Choosing $d_{max}$ in $G_{n+2}$ as the diagonal corresponding to
the segment $[2,k]$ where $k=\bigl[\frac{n+7}2\bigr]$ we observe
that it is also a maximal diagonal for $G_{n+3}$. Now take
$I'_1=\{d_{max}\}$ and take $I'_2$ to be the union of $J'_2$ and
all diagonals with endpoint 1 intersecting $d_{max}$. Since the
number of distinguished points on $d_{max}$ is
$\bigl[\frac{n+1}2\bigr]$, we obtain from~\eqref{ji}
and~\eqref{tildei}
\[
  |I'|=1+|I'_2|=1+|J'_2|+{\textstyle\bigl[\frac{n+1}2\bigr]}=
  |J'|+{\textstyle\bigl[\frac{n+1}2\bigr]}
  >|\tilde{I}|+{\textstyle\bigl[\frac{n+1}2\bigr]}\ge{|I|},
\]
which finishes the inductive argument.
\end{proof}

\begin{lemm}\label{le2}
Suppose $cc(P_I)=2$, $I=I_{1}\sqcup{I_{2}}$ and $|I|\ge{q+1}$.
Then either $|I_{1}|=1$ or $|I_{2}|=1$.
\end{lemm}
\begin{proof}
Assume the opposite, i.e. $|I_{1}|\ge{2}$ and $|I_{2}|\ge{2}$. By
Lemma~\ref{i1i2}, we may find another $I'=I'_1\sqcup I'_2$ such
that $|I'_{1}|=1$ and $|I'|>|I|\ge q+1$. On the other hand
$|I'_{1}|=1$ implies that $I'_{1}=\{d\}$ and $|I'|\le
1+p(d)\le1+q$. A contradiction.
\end{proof}

\begin{lemm}\label{le3}
Suppose $cc(P_I)=2$, $I=I_{1}\sqcup{I_{2}}$ and $|I|=q+1$. Then $I_{1}$ consists of a
single maximal diagonal $d_{max}$, and $I_{2}$ consists of all
diagonals of $G_{n+3}$ which intersect $d_{max}$.
\end{lemm}
\begin{proof}
By Lemma~\ref{le2}, we may assume that $I_{1}$ consists of a
single diagonal~$d$. Then
$$
  1+q=|I|=|I_{1}|+|I_{2}|\le 1+p(d)\le{1+q},
$$
which implies that $p(d)=q$ and $|I_{2}|=p(d)$.
\end{proof}

\begin{lemm}\label{le4}
Suppose $|I|>q+1$. Then $cc(P_I)=1$.
\end{lemm}
\begin{proof}
We have $|I|>q+1>l(n)$. Hence, $cc(P_I)\le2$ by
Lemma~\ref{le1}. Assume $cc(P_I)=2$ and $I=I_1\sqcup I_2$. Then
$|I_{1}|=1$ by Lemma~\ref{le2}, i.e. $I_{1}=\{d\}$ and $|I|\le
1+p(d)\le{1+q}$. This contradicts the assumption $|I|>q+1$.
\end{proof}
Now we can finish the induction in the proof of
Theorem~\ref{mainass}. From Corollary~\ref{fact1} and
Lemma~\ref{le3} we obtain that the number $\beta^{-q,2(q+1)}(P)$
is equal to the number of maximal diagonals in~$G_{n+3}$. The
latter equals $n+3$ when $n$ is even, and $\frac{n+3}{2}$ when $n$
is odd. The fact that $\beta^{-i,2(i+1)}(P)$ vanishes for
$i\ge{q+1}$ follows from Corollary~\ref{fact1} and
Lemma~\ref{le4}.
\end{proof}

We also calculate the bigraded Betti numbers of $\As^{n}$ for
$n\le 5$ using software package \emph{Macaulay 2}, see~\cite{Mac}.

The tables below have $n-1$ rows and $m-n-1$ columns. The number
in the intersection of the $k$th row and the $l$th column is
$\beta^{-l,2(l+k)}(\As^{n})$, where $ 1\le{l}\le{m-n-1}$ and
$2\le{l+k}\le{m-2}$. The other bigraded Betti numbers are zero
except for $\beta^{0,0}(\As^{n})=\beta^{-(m-n),2m}(\As^{n})=1$,
see~\cite[Ch.8]{B-P}. The bigraded Betti numbers given by
Theorem~\ref{mainass} are printed in bold.

\noindent{\bf1.} $n=2$, $m=5$.\\[2pt]
{\small
\begin{tabular}{|c|c|}
\hline
5 & 5\tabularnewline
\hline
\end{tabular}
}

\medskip


\noindent{\bf2.} $n=3$, $m=9$.\\[2pt]
{\small
\begin{tabular}{|c|c|c|c|c|}
\hline 15 & 35 & 24 & \bf{3} & \bf{0}\tabularnewline \hline \bf{0}
& \bf{3} & 24 & 35 & 15\tabularnewline \hline
\end{tabular}
}

\medskip

\noindent{\bf 3.} $n=4$, $m=14$.\\[2pt]
{\small
\begin{tabular}{|c|c|c|c|c|c|c|c|c|}
\hline 35 & 140 & 217 & 154 & 49 & \bf{7} & \bf{0} & \bf{0} &
\bf{0}\tabularnewline \hline 0 & 28 & 266 & 784 & 1094 & 784 & 266
& 28 & 0\tabularnewline \hline \bf{0} & \bf{0} & \bf{0} & \bf{7} &
49 & 154 & 217 & 140 & 35\tabularnewline \hline
\end{tabular}
}
\medskip

\noindent\textbf{4.} $n=5$, $m=20$.\\[2pt]
{\small
\begin{tabular}{|c|c|c|c|c|c|c|c|c|c|c|c|c|c|}
\hline
70 & 420 & 1089 & 1544 & 1300 & 680 & 226 & 44 & \bf{4} & \bf{0} & \bf{0}
\tabularnewline
\hline
0 & 144 & 1796 & 8332 & 20924 & 32309 & 32184 & 20798 & 8480 & 2053 & 264
\tabularnewline
\hline
0 & 0 & 12 & 264 & 2053 & 8480 & 20798 & 32184 & 32309 & 20924 & 8332
\tabularnewline
\hline
\bf{0} & \bf{0} & \bf{0} & \bf{0} & \bf{0} & \bf{4} & 44 & 226 & 680 & 1300 & 1544
\tabularnewline
\hline
\end{tabular} $\vdots$
}


\bigskip
The topology of moment-angle manifolds $\zp$ corresponding to
associahedra is far from being well understood even in the case
when $P$ is 3-dimensional. In this case the cohomology ring
$H^*(\zp)$ has nontrivial triple \emph{Massey products} by a
result of Baskakov (see~\cite[\S8.4]{B-P} or~\cite[\S5.3]{P}),
which implies that $\zp$ is not \emph{formal} in the sense of
rational homotopy theory.

\section{Truncation polytopes}
Let $P$ be a simple $n$-polytope and $v\in{P}$ its vertex. Choose
a hyperplane $H$ such that $H$ separates $v$ from the other
vertices and $v$ belongs to the positive halfspace $H_\ge$
determined by~$H$. Then $P\cap H_\ge$ is an $n$-simplex, and
$P\cap H_\le$ is a simple polytope, which we refer to as a
\emph{vertex cut} of~$P$. When the choice of the cut vertex is
clear or irrelevant we use the notation $\vc(P)$. We also use the
notation $\vc^k(P)$ for a polytope obtained from $P$ by iterating
the vertex cut operation $k$ times.

As an example of this procedure, we consider the polytope
$\vc^{k}(\Delta^{n})$, where $\Delta^{n}$ is an $n$-simplex,
$n\ge2$. We refer to  $\vc^{k}(\Delta^{n})$ as a \emph{truncation
polytope}; it has $m=n+k+1$ facets. Note that the combinatorial
type of $\vc^{k}(\Delta^{n}) $ depends on the choice of the cut
vertices if $k\ge{3}$, however we shall not reflect  this in the
notation.

Simplicial polytopes dual to $\vc^{k}(\Delta^{n})$ are known as
\emph{stacked polytopes}. They can be obtained from $\Delta^{n}$
by iteratively adding pyramids over facets.

The Betti numbers for stacked polytopes were calculated
in~\cite{T-H}, but the grading used there was different. We
include this result below, with a proof that uses a slightly
different argument and our `topological' grading and notation:
\begin{theo}\label{maintr}

Let $P=\vc^{k}(\Delta^{n})$ be a truncation polytope. Then for
$n\ge{3}$ the bigraded Betti numbers are given by the following
formulae:
\begin{align*}
&\beta^{-i,2(i+1)}(P)=i\binom{k+1}{i+1},\\
&\beta^{-i,2(i+n-1)}(P)=(k+1-i)\binom{k+1}{k+2-i},\\
&\beta^{-i,2j}(P)=0,\quad\text{for } i+1< j< i+n-1.
\end{align*}
The other bigraded Betti numbers are also zero, except for
\[
  \beta^{0,0}(P)=\beta^{-(m-n),2m}(P)=1.
\]
\end{theo}
\begin{rema}
The first of the above formulae was proved in~\cite{Ch-K} combinatorially.
\end{rema}
\begin{proof}
We start by analysing the behavior of bigraded Betti numbers under
a single vertex cut. Let $P$ be an arbitrary simple polytope and
$P'=\vc(P)$. We denote by $Q$ and $Q'$ the dual simplicial
polytopes respectively, and denote by $K$ and $K'$ their boundary
simplicial complexes. Then $Q'$ is obtained by adding a pyramid
with vertex $v$ over a facet $F$ of~$Q$. We also denote by $V$,
$V'$ and $V(F)$ the vertex sets of $Q$, $Q'$ and $F$ respectively,
so that $V'=V\cup{v}$.

The proof of the first formula is based on the following lemma:
\begin{lemm}\label{betapp'}
Let $P$ be a simple $n$-polytope with $m$ facets and $P'=\vc(P)$.
Then
$$
   \beta^{-i,2(i+1)}(P')=\binom{m-n}{i}+\beta^{-(i-1),2i}(P)+\beta^{-i,2(i+1)}(P).
$$
\end{lemm}
\begin{proof}

Applying Theorem~\ref{hoh} for $j=i+1$, we obtain:
\begin{align}
\beta^{-i,2(i+1)}(P')&=\sum\limits_{W\subset{V'},\,|W|=i+1}\dim\,\widetilde{H}_{0}
(K'_{W})\notag\\
&=\sum\limits_{W\subset{V'},\,v\in{W},\,|W|=i+1}\dim\,\widetilde{H}_{0}
(K'_{W})\label{sum1}\\
&+\sum\limits_{W\subset{V'},\,v\notin{W},\,|W|=i+1}\dim\,\widetilde{H}_{0}
(K'_{W})\label{sum2}.
\end{align}

Sum \eqref{sum2} above is $\beta^{-i,2(i+1)}(P)$ by
Theorem~\ref{hoh}.

For sum \eqref{sum1} we have: in $W$ there are $i$ `old' vertices
and one new vertex~$v$. Therefore, the number of connected
components of $K'_{W}$ (which is by 1 greater than the dimension of $\widetilde{H}_{0}
(K'_{W})$) either remains the same (if
$W\cap{F}\neq{\varnothing}$) or increases by 1 (if
$W\cap{F}=\emptyset$, in which case the new component is the new
vertex~$v$). The number of subsets $W$ of the latter type is equal
to the number of ways to choose $i$ vertices from the $m-n$ `old'
vertices that do not lie in~$F$. Sum~\eqref{sum1} is therefore
given by
$$
   \sum\limits_{W\subset{V},|W|=i}\dim\,\widetilde{H}_{0}(K_{W})
   +\binom{m-n}{i}=\beta^{-(i-1),2i}(P)+\binom{m-n}{i},
$$
where we used Theorem~\ref{hoh} again.
\end{proof}

Now the first formula of Theorem~\ref{maintr} follows by induction
on the number of cut vertices, using the fact that
$\beta^{-i,2(i+1)}(\Delta^n)=0$ for all~$i$ and
Lemma~\ref{betapp'}.\\

The second formula follows from the bigraded Poincare duality, see Theorem~\ref{Poinc}.\\

The proof of the third formula relies on the following lemma.

\begin{lemm}\label{zero}
Let $P$ be a truncation polytope, $K$ the boundary complex of the
dual simplicial polytope, $V$ the vertex set of~$K$, and $W$ a
nonempty proper subset of~$V$. Then
$$
   \widetilde{H}_{i}(K_{W})=0\quad\text{for }i\neq{0,n-2}.
$$
\end{lemm}
\begin{proof}

The proof is by induction on the number $m=|V|$ of vertices
of~$K$. If $m=n+1,$ then $P$ is an $n$-simplex, and $K_{W}$ is
contractible for every proper subset $W\subset{V}$.

To make the induction step we consider $V'=V\cup{v}$ and $V(F)$ as
in the beginning of the proof of Theorem~\ref{maintr}. Assume the
statement is proved for $V$ and let $W$ be a proper subset
of~$V'.$

We consider the following 5 cases.

\smallskip

\noindent\emph{Case 1:} $v\in{W},\; W\cap{V(F)}\neq{\varnothing}.$

If $V(F)\subset{W}$, then $K'_{W}$ is a subdivision of
$K_{W-\{v\}}$. It follows that
$\widetilde{H}_{i}(K'_{W})\cong{\widetilde{H}_{i}(K_{W-\{v\}})}$.

If $W\cap{V(F)}\neq{V(F)}$, then we have
$$
   K'_{W}=K_{W-\{v\}}\cup{K'_{W\cap{V(F)\cup{\{v\}}}}},\quad
   K_{W-\{v\}}\cap{K'_{W\cap{V(F)\cup{\{v\}}}}}=K_{W\cap{V(F)}},
$$
and both $K_{W\cap{V(F)}}$ and $K'_{W\cap{V(F)\cup{\{v\}}}}$ are
contractible. From the Mayer--Vietoris exact sequence we again
obtain
$\widetilde{H}_{i}(K'_{W})\cong{\widetilde{H}_{i}(K_{W-\{v\}})}$.

\smallskip

\noindent\emph{Case 2:} $v\in{W},\; W\cap{V(F)}=\varnothing.$

In this case it is easy to see that
$K'_{W}=K_{W-\{v\}}\sqcup{\{v\}}.$ It follows that
$$
   \widetilde{H}_{i}(K'_{W})\cong\begin{cases}
   \widetilde{H}_{i}(K_{W-\{v\}})\oplus{\k},&\text{for $i=0;$}\\
   \widetilde{H}_{i}(K_{W-\{v\}}),&\text{for $i>0.$}
                                     \end{cases}
$$

\smallskip

\noindent\emph{Case 3:}  $W=V'-\{v\}=V.$

Then $K'_{W}$ is a triangulated $(n-1)$-disk and therefore
contractible.

\smallskip

\noindent\emph{Case 4:} $v\not\in{W},\; V(F)\subset{W},\;
W\neq{V}.$

We have
$$
   K_{W}=K'_{W}\cup{F},\quad
   K'_{W}\cap{F}=\partial{F},
$$
where $\partial{F}$ is the boundary of the facet $F$. Since
$\partial{F}$ is a triangulated ${(n-2)}$-sphere and~$F$ is a
triangulated $(n-1)$-disk, the Mayer--Vietoris homology sequence
implies that
$$
   \widetilde{H}_{i}(K'_{W})\cong\begin{cases}
   \widetilde{H}_{i}(K_{W}),&\text{for $i<n-2;$}\\
   \widetilde{H}_{i}(K_{W})\oplus{\k},&\text{for $i=n-2.$}
                                     \end{cases}
$$

\smallskip

\noindent\emph{Case 5:} $v\not\in{W},\; V(F)\not\subset{W}.$ In
this case we have $K'_{W}\cong{K_{W}}.$

In all cases we obtain
$$
   \widetilde{H}_{i}(K'_{W})\cong\widetilde{H}_{i}(K_{W-\{v\}})=0\quad
   \text{for }0<i<n-2,
$$
which finishes the proof by induction.
\end{proof}

Now the third formula of Theorem~\ref{maintr} follows from
Theorem~\ref{hoh} and Lemma~\ref{zero}.

The last statement of Theorem~\ref{maintr} follows
from~\cite[Cor.~8.19]{B-P}.
\end{proof}

For the sake of completeness we include the calculation of the
bigraded Betti numbers in the case $n=2$, that is, when $P$ is a
polygon.

\begin{prop}
If $P=\vc^{k}(\Delta^{2})$ is an $(k+3)$-gon, then
\begin{align*}
   &\beta^{-i,2(i+1)}(P)=i\binom{k+1}{i+1}+(k+1-i)\binom{k+1}{k+2-i},\\
   &\beta^{0,0}(P)=\beta^{-(k+1),2(k+3)}(P)=1,\\
   &\beta^{-i,2j}(P)=0,\quad\text{otherwise}.
\end{align*}
\end{prop}
\begin{proof}
This calculation was done in~\cite[Example~8.21]{B-P}. It can be
also obtained by a Mayer--Vietoris argument as in the proof of
Theorem~\ref{maintr}.
\end{proof}

\begin{coro}
The bigraded Betti numbers of truncation polytopes
$P=vc^{k}(\Delta^{n})$ depend only on the dimension and the number
of facets of $P$ and do not depend on its combinatorial type.
Moreover the numbers $\beta^{-i, 2(i+1)}$ do not depend on the
dimension~$n$.
\end{coro}

The topological type of the corresponding moment-angle manifold
$\mathcal{Z}_P$ is described as follows:

\begin{theo}[{see \cite[Theorem 6.3]{B-M}}]
Let $P=\vc^{k}(\Delta^{n})$ be a truncation polytope. Then the
corresponding moment-angle manifold $\mathcal{Z}_P$ is
diffeomorphic to the connected sum of sphere products:
\[
  \mathop{\#}_{j=1}^{k}
  \bigl(S^{j+2}\times S^{2n+k-j-1}\bigr)^{\#j\binom{k+1}{j+1}},
\]
where $X^{\#k}$ denotes the connected sum of $k$ copies of~$X$.
\end{theo}

It is easy to see that the Betti numbers of the connected sum
above agree with the bigraded Betti numbers of $P$,
see~\eqref{bp}.



\begin{thebibliography}{99}
\bibitem{B-M}
Fr\'ed\'eric Bosio and Laurent Meersseman. \emph{Real quadrics in
$\mathbb C^n$, complex manifolds and convex polytopes.} Acta
Math.~\textbf{197} (2006), no.~1, 53--127.

\bibitem{B}
Victor M. Buchstaber. \emph{Lectures on toric
topology}. In \emph{Proceedings of Toric Topology Workshop KAIST
2008}. Trends in Math.~{\bf10}, no.~1. Information Center for
Mathematical Sciences, KAIST, 2008, pp.~1--64.

\bibitem{B-P}
Victor M. Buchstaber and Taras E. Panov. \emph{Torus Actions in
Topology and Combinatorics} (in Russian). MCCME, Moscow, 2004, 272
pages.

\bibitem{Ch-K}
Suyoung Choi and Jang Soo Kim. \emph{A combinatorial proof of a
formula for Betti numbers of a stacked polytope.} Electron. J.
Combin.~\textbf{17} (2010), no.~1, Research Paper~9, 8~pp.;
arXiv:math.CO/0902.2444.

\bibitem{Mac}
\emph{Macaulay 2}. A software system devoted to supporting
research in algebraic geometry and commutative algebra. Available
at {\tt http://www.math.uiuc.edu/Macaulay2/}

\bibitem{P}
Taras Panov. \emph{Cohomology of face rings, and torus actions},
in ``Surveys in Contemporary Mathematics''. London Math. Soc.
Lecture Note Series, vol.~\textbf{347}, Cambridge, U.K., 2008, pp.
165--201; arXiv:math.AT/0506526.

\bibitem{Pa}
Taras Panov. \emph{Moment--angle manifolds and complexes}.
In \emph{Proceedings of Toric Topology Workshop KAIST
2010}. Trends in Math.~{\bf12}, no.~1. Information Center for
Mathematical Sciences, KAIST, 2010, pp.~43--69.

\bibitem{S}
Richard P. Stanley. \emph{Combinatorics and Commutative Algebra},
second edition. Progr. in Math.~{\bf 41}. Birkh\"auser, Boston,
1996.

\bibitem{stas63} James D. Stasheff. \emph{Homotopy associativity of
H-spaces. I.} Transactions Amer. Math. Soc.~\textbf{108} (1963),
275--292.

\bibitem{T-H}
Naoki Terai and Takayuki Hibi. \emph{Computation of Betti numbers
of monomial ideals associated with stacked polytopes.} Manuscripta
Math., 92(4): 447--453, 1997.

\end{thebibliography}
\end{document}